\documentclass[10pt]{article}

\usepackage[title,toc,titletoc,page]{appendix}

\usepackage{amsfonts}
\usepackage{amsmath}
\usepackage{amsthm}
\usepackage{amssymb}
\usepackage{latexsym}
\usepackage[all]{xy}
\usepackage{xcolor}
\usepackage{multicol}

\newtheorem{theorem}{Theorem}[section]
\newtheorem{lemma}[theorem]{Lemma}
\newtheorem{proposition}[theorem]{Proposition}

\newtheorem{definition}[theorem]{Definition}

\theoremstyle{remark}
\newtheorem{remark}[theorem]{Remark}

\newcommand{\eps}{\epsilon}

\newcommand{\id}{\mathrm{id}}

\newcommand{\mbv}{\mathbf{v}}

\newcommand{\mcA}{\mathcal{A}}

\newcommand{\mcU}{\mathcal{U}}

\newcommand{\mfa}{\mathfrak{a}}

\newcommand{\mfb}{\mathfrak{b}}
\newcommand{\mfB}{\mathfrak{B}}

\newcommand{\mfc}{\mathfrak{c}}

\newcommand{\mfg}{\mathfrak{g}}
\newcommand{\mfgl}{\mathfrak{g}\mathfrak{l}}

\newcommand{\mfo}{\mathfrak{o}}

\newcommand{\mfp}{\mathfrak{p}}

\newcommand{\mfq}{\mathfrak{q}}

\newcommand{\mfsp}{\mathfrak{sp}}

\newcommand{\mfU}{\mathfrak{U}}

\newcommand{\msB}{\mathsf{B}}

\newcommand{\msc}{\mathsf{c}}

\newcommand{\msE}{\mathsf{E}}

\newcommand{\mss}{\mathsf{s}}
\newcommand{\mst}{\mathsf{t}}

\newcommand{\End}{\mathrm{End}}

\newcommand{\wt}{\widetilde}
\newcommand{\lra}{\longrightarrow}

\newcommand{\onto}{\twoheadrightarrow}

\newcommand{\C}{\mathbb{C}}

\newcommand{\Z}{\mathbb{Z}}

\newcommand{\ot}{\otimes}
\newcommand{\ol}{\overline}

\DeclareMathOperator{\sgn}{sgn}

\begin{document}

\setlength{\pdfpagewidth}{8.5in}
\setlength{\pdfpageheight}{11in}

\setlength{\parskip}{1em}

\begin{center}

\Large{\textbf{  Quantized enveloping superalgebra of type $P$ }}

Saber Ahmed, Dimitar Grantcharov, Nicolas Guay
\end{center}

\begin{abstract} We introduce a new quantized enveloping superalgebra $\mfU_q\mfp_n$ attached to the Lie superalgebra $\mfp_n$ of type $P$. The superalgebra $\mfU_q\mfp_n$ is a quantization of a Lie bisuperalgebra structure on $\mfp_n$ and we study some of its basic properties. We  also introduce the periplectic $q$-Brauer algebra and prove that it is the centralizer of the $\mfU_q\mfp_n$-module structure on $\C(n|n)^{\otimes l}$. We end by proposing a definition for a new periplectic $q$-Schur superalgebra.

\end{abstract}

\section*{Introduction}

The simple finite-dimensional Lie superalgebras over $\mathbb{C}$ were classified by V. Kac in \cite{Kac}. The list in \textit{loc. cit.} contains three classes of Lie superalgebras: basic, strange and Cartan-type. There are two types of strange Lie superalgebras - $P$ and $Q$ - both of which are interesting due to the algebraic, geometric, and combinatorial properties of their representations. The study of the representations of type $P$ Lie superalgebras, which are also called periplectic in the literature, has attracted considerable attention in the last  five years. Interesting results on the category $\mathcal O$, the associated periplectic Brauer algebras, and related theories have been established in \cite{BDEA^+1}, \cite{BDEA^+2}, \cite{CP}, \cite{Co}, \cite{CE1}, \cite{CE2}, \cite{DHIN}, \cite{EAS1}, \cite{EAS2}, \cite{HIR}, \cite{IN}, \cite{IRS}, \cite{KT}, \cite{Ser}, among others.

The purpose of this paper is to introduce a quantum superalgebra of type $P$ via the FRT formalism \cite{FRT}. A similar approach was used by G. Olshanski in \cite{Ol} to define quantum superalgebras of type $Q$. We prove that our quantized enveloping superalgebra $\mfU_q\mfp_n$ quantizes a Lie bisuperalgebra structure on $\mfp_n$, a periplectic Lie superalgebra. 

Using a Manin triple, we find a solution $s$ of the classical Yang-Baxter equation. This element is similar but different from the fake Casimir element used in \cite{BDEA^+1}, \cite{BDEA^+2}. The quantum version of $s$, denoted $S$, is a solution of the quantum Yang-Baxter equation which serves as an essential ingredient in the definition of $\mfU_q\mfp_n$. It follows that the tensor superspace $\mathbb{C}(n|n)^{\otimes \ell}$ is a representation of $\mfU_q\mfp_n$ and the centralizer of the action of $\mfU_q\mfp_n$ is a quantum version of the periplectic Brauer algebra. The classical setting corresponding to $q=1$ was studied in \cite{Mo}. A similar result for type $Q$ Lie superalgebras was established in \cite{Ol}, where the centralizer of the action of the quantized enveloping superalgebra was proven to be the Hecke-Clifford superalgebra of the symmetric group $S_{\ell}$.  Having at our disposal the periplectic $q$-Brauer algebra, we can introduce the periplectic $q$-Schur superalgebra in a natural way. We conjecture that these are mutual centralizers (that is, they satisfy a double-centralizer property).

One immediate problem is to define $\mfU_q\mfp_n$ in terms of Drinfeld-Jimbo generators and relations and study its category $\mathcal O$. For type $Q$ Lie superalgebras, this problem was addressed in \cite{GJKK}. Furthermore,  in \cite{GJKKK}, a theory of crystal bases for the tensor representations of $\mfU_q\mfg$ was established. Unfortunately, it is unlikely that natural crystal bases exist in the type $P$ case due to the nonsemisimplicity of the category of tensor modules, contrary to what happens in type $Q$. Another natural direction is to construct, using also the FRT formalism, quantum affine superalgebras of type $P$. (See \cite{ChGu} for the type $Q$ case.) Yangians of type $P$ and $Q$ appeared already many years ago in the work of M. Nazarov \cite{Na1,Na2}. We hope to return to these questions in a future publication.

After setting up the notation and basic definitions in the first section, we introduce the ``butterfly'' Lie bisuperalgebra in Section \ref{bisup} and define the quantized enveloping superalgebra of type $P$ in the following section. The main result of Section \ref{sec:Uq} is Theorem \ref{QYBE}, which states that $S$, the $q$-deformation of $s$, is a solution of the quantum Yang-Baxter equation. In Section \ref{limquant}, we prove that $\mfU_q\mfp_n$ is a quantization of the Lie bisuperalgebra structure from Section \ref{bisup}: see Theorem \ref{thm-lim-cobr}. The new periplectic $q$-Brauer algebra $\mfB_{q,\ell}$ and the new periplectic $q$-Schur algebra are introduced in the last section, where we prove that $\mfB_{q,\ell}$ can be defined equivalently either using generators and relations or as the centralizer of the action of $\mfU_q(\mfp_n)$ on the tensor space: see Theorem \ref{thm:Bqcent}.

The proofs of our results require extensive computations: further details for all the computations can be found in \cite{AGG}.

\noindent \textbf{Acknowledgements:} The second named author is partly supported by the Simons Collaboration Grant 358245. He also would like to thank the Max Planck Institute in Bonn (where part of this work was completed) for the excellent working conditions. The third named author gratefully acknowledges the financial support of the Natural Sciences and Engineering Research Council of Canada provided via the Discovery Grant Program. We thank Patrick Conner, Robert Muth, and Vidas Regelskis for help with certain computations in the preliminary stages of the present paper. We also thank Nicholas Davidson, Jonathan Kujawa for the useful discussions. Finally,  we are grateful to the  referees for the valuable suggestions. 

\section{The Lie superalgebra of type $P$}

Let $\C(n|n)$ be the vector superspace $\C^n \oplus \C^n$ spanned by the odd standard basis vectors $e_{-n}, \ldots, e_{-1}$ and the even standard basis vectors $e_1,\ldots,e_n$. Let $M_{n|n}(\C)$ be the  vector superspace consisting of matrices $A=(a_{ij})$ with $a_{ij} \in \C$ and with rows and columns labelled using the integers $-n,\ldots, -1,1,\ldots,n$, so $i,j \in \{ \pm 1, \pm 2, \ldots, \pm n  \}$. Set $p(i) =1 \in \Z_2$ if $-n \le i\le -1$ and $p(i)=0 \in \Z_2$ if $1\le i\le n$. The parity of the elementary matrix $E_{ij}$ is $p(i)+p(j) \; \mathrm{mod}\, 2$. We denote by $\mfgl_{n|n}$ the Lie superalgebra over $\C$ whose underlying vector space is $M_{n|n}(\C)$ and which is equipped with the Lie superbracket $$[E_{ij}, E_{kl}] = \delta_{jk}E_{il} - (-1)^{(p(i)+p(j))(p(k)+p(l))} \delta_{il} E_{kj}.$$ 

Recall that the supertranspose $(\cdot)^{\rm st}$ on $\mfgl_{n|n}$ is given by the formula $(E_{ij})^{\rm st} = (-1)^{p(i)(p(j)+1)} E_{ji}$. The involution $\iota$ on $\mfgl_{n|n}$ which will be relevant for this paper is given by $\iota(X) = -\pi(X^{\rm st})$ where $\pi: \mfgl_{n|n} \lra \mfgl_{n|n}$ is the linear map given by $\pi(E_{ij}) = E_{-i,-j}$.

\begin{definition}
The Lie superalgebra $\mfp_n$ of type $P$, which is also called the periplectic Lie superalgebra, is the subspace of fixed points of $\mfgl_{n|n}$ under the involution $\iota$, that is, $\mfp_n = \{ X \in \mfgl_{n|n} \, | \, \iota(X)=X \}$.
\end{definition} 

If $X \in \mfp_n$ with $\begin{pmatrix} A &  B \\ C & D  \end{pmatrix}$ and $A,B,C,D \in M_n(\C)$, then $D=-A^t$, $B=B^t$ and $C=-C^t$ where $t$ denotes the transpose with respect to the diagonal $i=-j$.  For convenience, we set
$$
\msE_{ij} = E_{ij} + \iota(E_{ij}) = E_{ij} - (-1)^{p(i)(p(j)+1)} E_{-j,-i}.
$$ The superbracket on $\mfp_n$ is given by \begin{eqnarray} [\msE_{ji}, \msE_{lk}] 
& = & \delta_{il} \msE_{jk} - (-1)^{(p(i)+p(j))(p(k)+p(l))} \delta_{jk} \msE_{li} \notag \\
& & - \delta_{i,-k}(-1)^{p(l)(p(k)+1)} \msE_{j,-l} - \delta_{-j,l}(-1)^{p(j)(p(i)+1)} \msE_{-i,k} \label{supbr}
\end{eqnarray}

A basis of $\mfp_n$ is provided by all the matrices $\msE_{ij}$ with indices $i$ and $j$ respecting one of the following inequalities: $$1 \le |j| < |i| \le n \text{ or } 1\le i=j \le n \text{ or } -n \le i=-j \le -1.$$
Note that  $\msE_{ij} = - (-1)^{p(i)(p(j)+1)}\msE_{-j,-i}$ for all $i,j\in\{\pm 1, \ldots, \pm n\}$,  hence $\msE_{i,-i} = 0$ when $1 \le i \le n$.

\section{Lie bisuperalgebra structure}\label{bisup}

To construct a Lie bisuperalgebra structure on $\mfp_n$, we define a Manin supertriple. We follow the idea in \cite{Ol} for the case of the Lie superalgebra of type $Q$.  Recall that a \emph{Manin supertriple} $(\mfa, \mfa_1, \mfa_2)$  consists of a Lie superalgebra $\mfa$ equipped with an ad-invariant supersymmetric non-degenerate bilinear form $\msB$ along with two Lie subsuperalgebras $\mfa_1, \mfa_2$ of $\mfa$ which are $\msB$-isotropic transversal subspaces of $\mfa$. Note that such a bilinear form $\msB$ defines a non-degenerate pairing between $\mfa_1$ and $\mfa_2$ and a supercobracket $\delta: \mfa_1 \to \mfa_1^{\otimes 2}$ via
$$
\msB^{\otimes 2} (\delta (X), Y_1 \otimes Y_2) = \msB (X, [Y_1,Y_2]),
$$
where $X \in \mfa_1$, $Y_1,Y_2 \in \mfa_2$.

\begin{definition}
The ``butterfly'' Lie superalgebra $\mfb_n$ is the subspace of $\mfgl_{n|n}$ spanned by $E_{ij}$ with $1\le |i| < |j|\le n$ and by $E_{ii} + E_{-i,-i}, E_{i,-i}$ for $1\le i\le n$.
\end{definition}

Note that after adding all diagonal matrices to $\mfb_n$ we obtain a Borel subalgebra of $\mfgl_{n|n}$ whose simple roots are all odd.
Note also that $\mfgl_{n|n} = \mfp_n \oplus \mfb_n$. It is well-known that the bilinear form $\msB(\cdot,\cdot)$ on $\mfgl_{n|n}$ given by the super-trace, $\msB(A,B) = \mathrm{Str}(AB)$, is ad-invariant, supersymmetric and non-degenerate.

One easily checks that $\msB(X_1,X_2) = 0$ if $X_1,X_2 \in \mfp_n$ or if $X_1,X_2 \in \mfb_n$. Hence we have the following result.
\begin{proposition} $(\mfgl_{n|n}, \mfp_n,\mfb_n)$ is a Manin supertriple.
\end{proposition}

\begin{remark} A similar Manin supertriple is given in \cite{LeSh}, \S2.2. \end{remark}

The quantum superalgebra that we will define in the next section will be a quantization of the Lie bisuperalgebra structure given by the Manin supertriple $(\mfgl_{n|n}, \mfp_n,\mfb_n)$.

We extend the form $\msB(\cdot,\cdot)$ to a non-degenerate pairing $\msB^{\otimes 2}$ on $\mfgl_{n|n} \otimes_{\C} \mfgl_{n|n}$ by setting $$\msB^{\otimes 2}(X_1 \ot X_2, Y_1 \ot Y_2) = (-1)^{|X_2||Y_1|} \msB(X_1,Y_1) \msB(X_2,Y_2)$$ for all homogeneous elements $X_1,X_2, Y_1,Y_2 \in \mfp_n$. The sign $(-1)^{|X_2||Y_1|}$ is necessary to make this form ad-invariant. 

Let 
\begin{equation} \label{equ-s}
   \mss = \sum\limits_{1\leq \vert j \vert < \vert i \vert \leq n} (-1)^{p(j)}\msE_{ij}\otimes E_{ji} + \frac{1}{2}\sum\limits_{1\leq i \leq n}\msE_{ii}\otimes(E_{ii}+E_{-i,-i})+\frac{1}{2}\sum\limits_{1\leq i \leq n}\msE_{-i,i}\otimes E_{i,-i}
\end{equation}
\begin{remark}
We note that the fake Casimir used in  \cite{BDEA^+1} is also defined using the sum of tensor product of basis vectors in $\mfp_n$ and their duals in $\mfp_n^{\perp}$, but the fake Casimir differs from the element $\mss$ defined above. One crucial difference is that the  space $\mfp_n^{\perp}$ used in \cite{BDEA^+1} is not a subalgebra of $\mfgl_{n|n}$, while $\mfb_n$ is.
\end{remark}

\begin{proposition} \label{prop-cybe} $\mss$ is a solution of the classical Yang-Baxter equation: $[\mss_{12}, \mss_{13}] + [\mss_{12}, \mss_{23}] + [\mss_{13}, \mss_{23}]=0$. 
\end{proposition}

The proof of the above proposition follows from the lemma below, which should be well-known among experts. 

\begin{lemma}\label{CYBE}
Let $\mfp$ be a finite dimensional Lie superalgebra and suppose that $(\mfp, \mfp_1, \mfp_2)$ is a Manin triple with respect to a certain supersymmetric, invariant, bilinear form $\msB(\cdot, \cdot)$. Let $\{ X_{i}\}_{i \in I}, \{ X'_i \}_{i\in I }$ be bases of $\mfp_1$ and $\mfp_2$, respectively, dual in the sense that $\msB(X_i', X_j) = \delta_{ij} $.  Set $s= \sum_{i\in I} X_i \ot X_i'$. Then $s$ is a solution of the classical Yang-Baxter equation.
\end{lemma}

We next compute the supercobracket $\delta$ using the identity $\msB(X, [Y_1,Y_2]) = \msB(\delta(X), Y_1 \otimes Y_2)$ for all $X\in\mfp_n$ and all $Y_1,Y_2 \in \mfb_n$.   
The formula for $\delta$ is (assuming, without loss of generality, that $|j|\le |i|$): 
 \begin{eqnarray} \delta(\msE_{ij}) & = &  \sum_{\stackrel{k=-n}{|j| < |k| < |i|}}^n (-1)^{p(k)+1} \big( \msE_{ik} \ot \msE_{kj} - (-1)^{(p(i)+p(k))(p(j)+p(k))} \msE_{kj} \ot \msE_{ik} \big) \notag  \\ 
 & & - \frac{1}{2} ((-1)^{p(i)}\msE_{ii} - (-1)^{p(j)} \msE_{jj}) \ot \msE_{ij} + \frac{1}{2} \msE_{ij} \ot ((-1)^{p(i)}\msE_{ii} - (-1)^{p(j)} \msE_{jj}) \notag  \\
 & & - \frac{\delta(i<0)}{2} \left(\msE_{i,-i}  \ot \msE_{-i,j} - (-1)^{p(j)}\msE_{-i,j} \ot \msE_{i,-i} \right) \label{cobr} \\
 & & + \frac{\delta(j>0)}{2} \left( (-1)^{p(i)}\msE_{-j,j} \ot \msE_{i,-j} + \msE_{i,-j} \ot \msE_{-j,j} \right)  \notag 
\end{eqnarray}

Finally, the super cobracket on $\mfp_n$ is related to the element $\mss$. The following lemma is standard.

\begin{lemma}
The super cobracket can also be expressed as \begin{equation} \delta(X) = [X\ot 1 + 1 \ot X,  \mss],\label{deltas}
 \end{equation}
 for $X \in \mfp_n$.
\end{lemma}

\section{Quantized enveloping superalgebra}\label{sec:Uq}

In this section, we define the quantized enveloping superalgebra $\mfU_q\mfp_n$ following the approach used in \cite{FRT} and  \cite{Ol}. We use a solution $S$ of the quantum Yang-Baxter equation such that $\mss$ is the classical limit of $S$.

For simplicity,  denote  by $\C_q$ the field $\C(q)$ of rational functions in the variable $q$ and set $\C_q(n|n) = \C_q \ot_{\C} \C(n|n)$.

\begin{definition}\label{Sq}
Let $S \in \End_{\C_q}(\C_q(n|n)^{\ot 2})$ be given by the formula:
\begin{eqnarray}
S & = & 1 + \sum_{1\le i \le n} \big( (q-1)E_{ii} + (q^{-1}-1) E_{-i,-i} \big) \ot (E_{ii} + E_{-i,-i}) + \frac{q-q^{-1}}{2} \sum_{-n \le i\le -1} \msE_{i,-i} \ot E_{-i,i} \notag \\
& & +(q-q^{-1}) \sum_{1\le |j| < |i| \le n} (-1)^{p(j)} \msE_{ij} \ot E_{ji} \label{Smat}
\end{eqnarray}
\end{definition}

\begin{remark} \label{q-to-h} If we define $S$ instead as an element of $\End_{\C[[\hbar]]}(\C_{\hbar}(n|n)^{\ot 2})$ by the same formula as in definition \ref{Sq} but with $q,q^{-1}$ replaced by $e^{\hbar/2}, e^{-\hbar/2}$ and $\C_q(n|n)^{\ot 2}$ replaced by $\C_{\hbar}(n|n)^{\ot 2}$, which equals $\C(n|n)^{\ot 2}[[\hbar]]$, then $S=1+\hbar\mss+O(\hbar^2)$. 
\end{remark}

\begin{theorem} \label{QYBE}
$S$ is a solution of the quantum Yang-Baxter equation: $S_{12} S_{13} S_{23} = S_{23} S_{13} S_{12}$.
\end{theorem}

\begin{proof}  The proof consists of verifying long computations. To simplify them, we have used the following method.  Set $f(q)= S_{12}S_{13}S_{23} - S_{23}S_{13}S_{12}$. The main idea is to consider $f(q)$ as a Laurent polynomial $\sum_{i=-3}^3 f_iq^i$ with coefficients $f_i$ in $\End_{\mathbb{C}}\left(\C_{n\mid n}^{\otimes 3}\right)$. Then one shows the eight relations $f(a) = 0$, $f'(b) = 0$, $f''(c) = 0$ for $a,b,c=\pm 1$ and $b=\pm\sqrt{-1}$. (Actually, just seven of those are enough.) We can then deduce that $f(q)$ is a scalar multiple of $(q-q^{-1})^3$ and we show that the coefficient of $q^3$ in $f(q)$ is zero.  

Here are some more details. 

Let us set  
$$
 C = \sum\limits_{1\leq i \leq n} (E_{ii}+E_{-i,-i})\otimes(E_{ii}+E_{-i,-i}).
$$
Then 
\begin{equation*}
    S = 1 + (q-q^{-1})\mathsf{s} + \left(\frac{q+q^{-1}}{2}-1\right)C.
\end{equation*}

For convenience, we introduce the following notation:
\begin{align*}
[\mathsf{s}C] = {} & \mathsf{s}_{12}C_{13} + \mathsf{s}_{12}C_{23} + \mathsf{s}_{13}C_{23} + C_{12}\mathsf{s}_{13} + C_{12}\mathsf{s}_{23} + C_{13}\mathsf{s}_{23} \\
{} & - \mathsf{s}_{23}C_{13} - \mathsf{s}_{23}C_{12} - \mathsf{s}_{13}C_{12} - C_{23}\mathsf{s}_{13} - C_{23}\mathsf{s}_{12} - C_{13}\mathsf{s}_{12}\\
[\mathsf{s}CC] = {} & \mathsf{s}_{12}C_{13}C_{23} + C_{12}\mathsf{s}_{13}C_{23} + C_{12}C_{13}\mathsf{s}_{23} - \mathsf{s}_{23}C_{13}C_{12} - C_{23}S_{13}C_{12} - C_{23}C_{13}\mathsf{s}_{12}\\
[\mathsf{s}\mathsf{s}C] = {} & \mathsf{s}_{12}\mathsf{s}_{13}C_{23} + C_{12}\mathsf{s}_{13}\mathsf{s}_{23} + \mathsf{s}_{12}C_{13}\mathsf{s}_{23} - \mathsf{s}_{23}\mathsf{s}_{13}C_{12} - C_{23}\mathsf{s}_{13}\mathsf{s}_{12} - \mathsf{s}_{23}C_{13}\mathsf{s}_{12}
\end{align*}

The relations $f(a) = 0$, $f'(b) = 0$, $f''(c) = 0$ for $a,b,c=\pm 1$ and $b=\pm\sqrt{-1}$ follow from the next two lemmas and checking these involves explicit computations.

\begin{lemma} \label{lem-1}
$[\mathsf{s}C] = 2[\mathsf{s}CC]$
\end{lemma}

\begin{lemma} \label{lem-2}
$[\mathsf{s}\mathsf{s}C] = 0$
\end{lemma}

For instance,  $f'(-1)=0$ follows from $f'(-1) = -4[\mathsf{s} C] + 8[\mathsf{s} CC]$ and the two lemmas. Furthermore, $$f^{\prime\prime}(-1) =  -4[\mathsf{s}C] + 8[\mathsf{s}CC] - 16[\mathsf{s}\mathsf{s}C] + 8([\mathsf{s}_{12},\mathsf{s}_{13}] + [\mathsf{s}_{12},\mathsf{s}_{23}] + [\mathsf{s}_{13},\mathsf{s}_{23}]).$$ Therefore, $f^{\prime\prime}(-1)=0$ thanks to Lemmas \ref{CYBE}, \ref{lem-1}, and \ref{lem-2}.  Similarly, the two lemmas above imply that $$f^{\prime}(\sqrt{-1}) = 2\sqrt{-1}[\mathsf{s}C] - 4\sqrt{-1}[\mathsf{s}CC] - 4[\mathsf{s}\mathsf{s}C]$$  vanishes.

The last step in the proof of Theorem \ref{QYBE} is to show the vanishing of the coefficient $f_3$ of $q^3$. We have $$f_3 = \mss_{12}\mss_{13}\mss_{23} - \mss_{23}\mss_{13}\mss_{12} + \frac{1}{4}[sCC] + \frac{1}{2}[ssC] + \frac{1}{8}C_{12}C_{13}C_{23} - \frac{1}{8}C_{23}C_{13}C_{12},$$ which simplifies to \begin{equation} \mss_{12}\mss_{13}\mss_{23} - \mss_{23}\mss_{13}\mss_{12} + \frac{1}{4}[sCC] \label{q3} \end{equation} thanks to Lemma \ref{lem-2} and $C_{12}C_{13}C_{23} - C_{23}C_{13}C_{12}=0$. Verifying that \eqref{q3} vanishes  follows by direct and extensive computations. \end{proof}

With the aid of $S$, we can now define the main object of interest in this paper.

\begin{definition}
The \emph{quantized enveloping superalgebra of $\mfp_n$} is the $\Z_2$-graded $\C_q-$algebra $\mfU_q\mfp_n$ generated by elements $t_{ij},t_{ii}^{-1}$ with $1 \le |i|\le |j|\le n$ and $i,j\in \{ \pm 1,\ldots, \pm n  \}$ which satisfy the following relations: \begin{equation} t_{ii} = t_{-i,-i},\,\, t_{-i,i}=0 \text{ if } i>0,\,\, t_{ij}=0 \text{ if } |i|>|j|;  \end{equation} \begin{equation} T_{12} T_{13} S_{23} = S_{23} T_{13} T_{12} \label{rttf} \end{equation} where $T = \sum_{|i|\le |j|} t_{ij} \ot_{\C} E_{ij}$ and the last equality holds in $\mfU_q\mfp_n \ot_{\C(q)} \End_{\C(q)}(\C_q(n|n))^{\ot 2}$. The $\Z_2$-degree of $t_{ij}$ is $p(i)+p(j)$.
\end{definition}

\begin{remark} \label{rem-odd-zero}
One immediate corollary of the definition above is that if $t_{ij}$ is odd, then $t_{ij}^2=0$. This follows for example after taking $i =k$ and $j = l$ in \eqref{exprel}.
\end{remark}

$\mfU_q\mfp_n$ is a Hopf algebra with antipode given by $T \mapsto T^{-1}$ and with coproduct given by $$\Delta(t_{ij})=\sum_{k=-n}^n (-1)^{(p(i) + p(k))(p(k)+p(j))} t_{ik} \ot t_{kj}.$$

\section{Limit when $q\mapsto 1$ and quantization}\label{limquant}

We want to explain how $\mfU\mfp_n$ can be viewed as the limit when $q\mapsto 1$ of $\mfU_q\mfp_n$ and how the co-Poisson Hopf algebra structure on $\mfU\mfp_n$, which is inherited from the cobracket $\delta$ on $\mfp_n$, can be recovered from the coproduct on $\mfU_q\mfp_n$. 

Set $\tau_{ij} = \frac{t_{ij}}{q-q^{-1}}$ if $i\neq j$ and set $\tau_{ii} = \frac{t_{ii}-1}{q-1}$. Let $\mcA$ be the localization of $\C[q,q^{-1}]$ at the ideal generated by $q-1$. Let $\mfU_{\mcA}\mfp_n$ be the $\mcA$-subalgebra of $\mfU_q\mfp_n$ generated by $\tau_{ij}$ when $1\le |i| \le |j| \le n$.

\begin{theorem}\label{defprop}
The map $\psi: \mfU\mfp_n \lra \mfU_{\mcA} \mfp_n/(q-1)\mfU_{\mcA}\mfp_n$ given by $\psi(\msE_{ji}) = (-1)^{p(j)} \ol{\tau}_{ij}$ for $|i|<|j|$, $1\le i=j\le n$,  and $\displaystyle \psi(\msE_{-i,i}) =-2\overline{\tau}_{i,-i}$ for $1\leq i \leq n$, is an associative $\C$-superalgebra  isomorphism.
\end{theorem}

\begin{proof}
First, we need to write down explicitly the defining relation \eqref{rttf}. Comparing coefficients of $E_{ij} \ot E_{kl}$ on both sides of relation \eqref{rttf}, we obtain: 
\begin{equation} 
\begin{split}  
& (-1)^{(p(i)+p(j))(p(k)+p(l))} t_{ij}t_{kl} - t_{kl}t_{ij} + \theta(i,j,k) \big( \delta_{|j|<|l|} - \delta_{|k|<|i|} \big) \eps t_{il}t_{kj} \\
& \qquad + (-1)^{(p(i)+p(j))(p(k)+p(l))} \big( \delta_{j>0}(q-1) + \delta_{j<0}(q^{-1}-1) \big) \big(\delta_{jl} + \delta_{j,-l} \big) t_{ij}t_{kl} \\
& \qquad  \qquad  - \big( \delta_{i>0}(q-1) + \delta_{i<0}(q^{-1}-1) \big) \big( \delta_{ik} + \delta_{i,-k} \big) t_{kl}t_{ij} \label{exprel} \\
&  \qquad  \qquad  \qquad + \theta(i,j,k) \delta_{j>0}  \delta_{j,-l}  \eps t_{i,-j}t_{k,-l} - (-1)^{p(j)} \delta_{i<0} \delta_{i,-k}  \eps t_{-k,l}t_{-i,j} \\
 & \qquad  + (-1)^{p(j)(p(i)+1)} \eps \sum_{-n \le a \le n} \big( (-1)^{p(i)p(a)}\theta(i,j,k) \delta_{j,-l} \delta_{|a|<|l|} t_{i,-a}t_{ka} + (-1)^{p(-j)p(a)}\delta_{i,-k} \delta_{|k|<|a|} t_{al}t_{-a,j}\big) \\
& \qquad  = 0  
\end{split} 
\end{equation}
In the identity above, we set
$$
\theta(i,j,k) = \sgn (\sgn(i)+\sgn(j)+\sgn(k)) \text{ and } \eps=q-q^{-1}.
$$

In order to check that  $\psi( [\msE_{ji}, \msE_{kl}])=[\psi(\msE_{ji}), \psi(\msE_{kl})]$, we proceed as follows. We apply $\psi$ on both sides of  \eqref{supbr}. To show that the resulting right hand side coincides with  $[\psi(\msE_{ji}), \psi(\msE_{kl})]$, we use \eqref{exprel} and pass to the quotient $ \mfU_{\mcA} \mfp_n/(q-1)\mfU_{\mcA}\mfp_n$. This is done via a long case-by-case verification for $i,j,k,l$. 

From the way $\mfU_{\mcA}\mfp_n$ is defined, it follows that $\psi$ is surjective. It remains to prove that it is injective. Since $S$ is a solution of the quantum Yang-Baxter equation, the space $\C_q(n|n)$ is a representation of $\mfU_q\mfp_n$ via the assignment $t_{ij}\mapsto s_{ij}$ (where $S = \sum_{i,j=-n}^n s_{ij} \otimes E_{ij}$), hence also of $\mcU_{\mcA}\mfp_n$ by restriction. More explicitly, $$\tau_{ij} \mapsto (-1)^{p(i)}\msE_{ji} \text{ if $|i|<|j|$},\text{ and } \tau_{i,-i} \mapsto E_{-i,i},\, \tau_{ii} \mapsto (E_{ii} - q^{-1} E_{-i,-i}) \text{ if $1\le i \le n$.}$$ Set $\C_{\mcA}(n|n) = \mcA\ot_{\C}\C(n|n)$. The  space $\C_{\mcA}(n|n)$  is a $\mfU_{\mcA}\mfp_n$-submodule and so are all the tensor powers $\C_{\mcA}(n|n)^{\ot\ell}$. We thus have a superalgebra homomorphism $\phi_{\ell}:\mcU_{\mcA}\mfp_n \lra \End_{\mcA}(\C_{\mcA}(n|n)^{\ot\ell})$ for each $\ell\ge 1$.

Let $\pi_{\ell}$ be the quotient homomorphism $$\End_{\mcA}(\C_{\mcA}(n|n)^{\ot\ell}) \lra \End_{\mcA}(\C_{\mcA}(n|n)^{\ot\ell})/(q-1)\End_{\mcA}(\C_{\mcA}(n|n)^{\ot\ell})\cong \End_{\C}(\C(n|n)^{\ot\ell}).$$ The composite $\pi_{\ell}\circ\phi_{\ell}$ descends to a homomorphism $\ol{\pi_{\ell}\circ\phi_{\ell}}$ from  $\mfU_{\mcA} \mfp_n/(q-1)\mfU_{\mcA}\mfp_n$ to $\End_{\C}(\C(n|n)^{\ot\ell})$. The composite $\ol{\pi_{\ell}\circ\phi_{\ell}}\circ\psi$ is the superalgebra homomorphism $\mfU\mfp_n\lra \End_{\C}(\C(n|n)^{\ot\ell})$ induced by the natural $\mfp_n$-module structure on $\C(n|n)^{\ot\ell}$ twisted by the automorphism of $\mfp_n$ given by $\msE_{ij}\mapsto (-1)^{p(i)+p(j)} \msE_{ij}$.

We can combine the homomorphisms $\ol{\pi_{\ell}\circ\phi_{\ell}}\circ\psi$ for all $\ell\ge 1$ to obtain a homomorphism $\mcU\mfp_n \lra \prod_{\ell=1}^{\infty} \End_{\C}(\C(n|n)^{\ot\ell})$. This map is injective since $\C(n|n)$ is a faithful representation of $\mfp_n$. It follows that $\psi$ is injective as well.  \end{proof}

We next show that a PBW-type theorem holds for $\mfU_q\mfp_n$. For this, we first  introduce a total  order $\prec$ on the set of generators $t_{ij}$, $1 \leq |i|\leq |j| \leq n$, of $\mfU_q\mfp_n$ as follows.
We declare that $t_{ij} \prec t_{kl}$ if
\begin{itemize}
\item[(i)]  $\vert i \vert > \vert k \vert$, or
\item[(ii)] $\vert i \vert = \vert k \vert $ and $\vert j \vert > \vert l \vert$, or
\item[(iii)] $i = k$ and $j=-l>0$, or 
\item[(iv)] $i=-k>0$ and $|j|=|l|$.
\end{itemize}
This order leads to a total lexicographic order on the set of words formed by the generators $t_{ij}$.  Namely, if $A=A_1 \cdots A_r$ and $B=B_1 \cdots B_s$ are two such words in the sense that each $A_k$ for $1\le k\le r$ and each $B_l$ for $1\le l\le s$ is equal to some generator $t_{ij}$, then $A\prec B$ if $r<s$ or if $r=s$ and there is a $p$ such that $A_k = B_k$ for $1\le k\le p-1$ and $A_p \prec B_p$. Note that, in this order,  the generators $t_{ij}$ with $i=j$ or $i=-j$ are not grouped together.  We call a generator of the from $t_{ii}$ \emph{diagonal}. Also, a word $A_1^{k_1}...A_r^{k_r}$  in the generators $t_{ij}$  is called a \emph{reduced monomial} if  $A_1 \prec \cdots \prec A_r $, and $k_i \in {\mathbb Z}_{>0}$ if $A_i$ is not diagonal, $k_i \in {\mathbb Z}\setminus\{0\}$ if $A_i$ is diagonal, and $k_i=1$ if $A_i$ is odd.

\begin{theorem}
The reduced monomials form a basis of $\mfU_q\mfp_n$ over $\C_q$.
\end{theorem}

\begin{proof} We first show that the set of  reduced monomials spans $\mfU_q\mfp_n$. Note that it is enough to show that all quadratic monomials are in the span of this set.  Let $t_{ij}t_{kl}$ be a quadratic monomial which is not reduced. We have that either $t_{kl} \neq t_{ij}$, or $i = k, \, j=l$ and $t_{ij}$ is odd. In the latter case, as explained in Remark \ref{rem-odd-zero}, $t_{ij}^2$=0. In the former case, we proceed with a case-by-case reasoning considering seven mutually exclusive subcases:
\begin{itemize}
\item[(a)] $|i| < |k|$ and $|j|\neq |l|$.
\item[(b)] $|i|<|k|$ and  $j=l$.
\item[(c)] $|i|<|k|$ and  $j=-l$. 
\item[(d)] $|i|=|k|$ and $|j| < |l|$.
\item[(e)] $i=k$  and $j=-l<0$. 
\item[(f)] $i=-k<0$ and $j=l$. 
\item[(g)] $i=-k<0$ and $j=-l$.
\end{itemize}

Let's consider in some details subcase (c). The remaining subcases are handled in a similar manner. In subcase (c),  (\ref{exprel}) simplifies to:
\begin{equation} \label{eq-subcase-iii}
\begin{split} 
& (-1)^{(p(i)+p(j))(p(k)+p(-j))}\big( \delta_{j>0}q + \delta_{j<0}q^{-1} \big)  t_{ij}t_{k,-j} - t_{k,-j}t_{ij} + \theta(i,j,k) \delta_{j>0} \eps t_{i,-j}t_{kj}  \\
 & \qquad  + (-1)^{p(j)(p(i)+1)} \eps \sum_{-n \le a \le n} (-1)^{p(i)p(a)}\theta(i,j,k)  \delta_{|a|<|j|} t_{i,-a}t_{ka} = 0  
\end{split} 
\end{equation}
Let us  assume that $|l|=|j|=1$. Then the previous equation reduces to \begin{equation*}
(-1)^{(p(i)+p(j))(p(k)+p(-j))}\big( \delta_{j>0}q + \delta_{j<0}q^{-1} \big)  t_{ij}t_{k,-j}  + \theta(i,j,k) \delta_{j>0} \eps t_{i,-j}t_{kj} = t_{k,-j}t_{ij} 
\end{equation*}
Replacing $j$ by $-j$ leads to the equation 
\begin{equation*}
(-1)^{(p(i)+p(-j))(p(k)+p(j))}\big( \delta_{j<0}q + \delta_{j>0}q^{-1} \big)  t_{i,-j}t_{kj}  + \theta(i,-j,k) \delta_{j<0} \eps t_{ij}t_{k,-j} = t_{kj}t_{i,-j} 
\end{equation*}
The monomials $t_{k,-j}t_{ij}$ and $t_{kj}t_{i,-j}$ are properly ordered and the previous two equations can be solved to express $t_{ij}t_{k,-j}$ and $ t_{i,-j}t_{kj}$ in terms of the former.

We then proceed by descending induction on $|j|$ and show that $t_{ij}t_{k,-j}$ can be expressed as a linear combination of properly ordered monomials. The base case $|j|=1$ was completed above. We use again (\ref{eq-subcase-iii}) and the corresponding equation obtained after switching $j$ and $-j$. In these two equations, by induction, the monomials $t_{i,-a}t_{ka}$ with $|a|<|j|$ can be expressed as linear combinations of properly ordered monomials. Moreover, $t_{k,-j}t_{ij}$ and $t_{kj}t_{i,-j}$ are already correctly ordered. As in the case $|l|=|j|=1$, we can then solve those two equations to express $t_{ij}t_{k,-j}$ and $ t_{i,-j}t_{kj}$ in terms of properly ordered monomials.

It remains to show that the reduced monomials form a linearly independent set. We follow the approach in \cite{Ol}. Let $M_1,\hdots, M_r$ be pairwise distinct reduced monomials in the generators $\tau_{ij}$ such that $a_1M_1 + \hdots + a_rM_r = 0$ for some $a_1,\hdots,a_r\in \C_q$. Without loss of generality, we can assume that $a_i\in\mcA$. It is sufficient to prove that $a_1,...,a_r \in \mcA$ implies $a_1,...,a_r \in (q-1)\mcA$.

Recall that there is a surjective homomorphism $\theta: \mfU_{\mathcal A} \mfp_n \to \mfU \mfp_n$ More precisely, $\theta$ is the composite of $\psi^{-1}$ from Theorem \ref{defprop} and the projection $\mfU_{\mathcal A} \mfp_n\to \mfU_{\mathcal A}\mfp_n/(q-1)\mfU_{\mathcal A}\mfp_n$ from Theorem \ref{defprop}. Let $\overline{M}_i = \theta(M_i)$ and denote by $\bar{a}_i$ the image of $a_i$ in $\mcA/(q-1)\mcA$. Since $M_1,\hdots, M_r$ are pairwise distinct reduced monomials, $\overline{M}_1,\hdots, \overline{M}_r$ are pairwise distinct monomials in $\mfU \mfp_n$. Then using that 
$$
\bar{a}_1 \overline{M}_1 + \ldots + \bar{a}_r \overline{M}_r  = \theta  (a_1M_1 + \hdots + a_rM_r ) = 0
$$
and the (classical) PBW Theorem for $\mfU \mfp_n$, we obtain $\bar{a}_1 = \hdots = \bar{a}_r = 0$. Hence $a_1,\hdots,a_r\in (q-1)\mcA$ as needed.
\end{proof}

As mentioned in Remark \ref{q-to-h}, we may replace $\C(q)$ by $\C((\hbar))$, $q$ by $e^{\hbar/2}$, and $\mcA$ by $\C[[\hbar]]$, and an analog of Theorem \ref{defprop} would hold true, implying that $\mfU_{\C[[\hbar]]} \mfp_n$ is a flat deformation of $\mfU\mfp_n$. Moreover, the next theorem states that $\mfU_{\C[[\hbar]]} \mfp_n$ is a quantization of the co-Poisson Hopf superalgebra structure on $\mfU\mfp_n$ induced by the Lie bisuperalgebra structure defined in Section \ref{bisup}. To be precise, the cobracket $\delta$ on $\mfp_n$ extends to a Poisson co-bracket on $\mfU\mfp_n$, which we also denote by $\delta$. Let $(\cdot)^{\circ}$ be the involution on $(\mfU_{\C[[\hbar]]}\mfp_n)^{\ot 2}$ given by $A_1 \ot A_2 \mapsto (-1)^{p(A_1)p(A_2)} A_2 \ot A_1$ where $p(A_i)$ is the $\Z/2\Z$-degree of $A_i,i=1,2$. 

For convenience, for $A \in \mfU_{\C[[\hbar]]}\mfp_n$, we denote by  $\ol{A}$ both the image of $A$ in $\mfU_{\C[[\hbar]]}\mfp_n/h\mfU_{\C[[\hbar]]}\mfp_n$ and the corresponding element in $\mfU\mfp_n$ via the isomorphism of the $\hbar$-analogue of Theorem \ref{defprop}. Similarly, we identify the corresponding elements in $\left( \mfU_{\C[[\hbar]]}\mfp_n/h\mfU_{\C[[\hbar]]}\mfp_n \right) \otimes \left( \mfU_{\C[[\hbar]]}\mfp_n/h\mfU_{\C[[\hbar]]}\mfp_n \right) $ and $\mfU\mfp_n \otimes \mfU\mfp_n$. 

\begin{theorem} \label{thm-lim-cobr}
If $A \in \mfU_{\C[[\hbar]]}\mfp_n$,  we have  $\ol{\hbar^{-1} (\Delta(A) - \Delta(A)^{\circ})} = \delta(\ol{A})$. Hence, $\mfU_{\C[[\hbar]]} \mfp_n$ is a quantization of the co-Poisson Hopf superalgebra structure on $\mfU\mfp_n$.
\end{theorem}
\begin{proof}

We show that the identity above holds for the generators $\tau_{ij}$ of $\mathfrak{U}_{\mathbb{C}[[\hbar]]}\mathfrak{g}$, so let $A=\tau_{ij}$.  We first note that the identity is trivially satisfied for $i=j$, as both sides are zero. Assume henceforth that $i\neq j$. Then: 
\begin{align*}
\hbar^{-1}\left(\Delta(\tau_{ij})-\Delta(\tau_{ij})^\circ\right) = {} & \left(\frac{e^{\hbar/2}-e^{-\hbar/2}}{\hbar}\right) \sum_{\substack{k=-n \\ \vert i \vert < \vert k \vert < \vert j \vert}}^n\left((-1)^{(p(i)+p(k))(p(j)+p(k))}\tau_{ik}\otimes \tau_{kj} - \tau_{kj}\otimes \tau_{ik} \right)\\ 
&  + \left(\frac{e^{\hbar/2}-1}{\hbar}\right) \left(\tau_{ii}\otimes \tau_{ij} - \tau_{ij}\otimes \tau_{ii} + \tau_{ij}\otimes \tau_{jj} - \tau_{jj}\otimes \tau_{ij} \right)\\ 
&  -  \left(\frac{e^{\hbar/2}-e^{-\hbar/2}}{\hbar}\right) \delta_{i>0}\left( (-1)^{p(j)}\tau_{i,-i}\otimes \tau_{-i,j} + \tau_{-i,j}\otimes \tau_{i,-i}   \right) \\ 
&  + \left(\frac{e^{\hbar/2}-e^{-\hbar/2}}{\hbar}\right)\delta_{j<0}\left( (-1)^{p(i)}\tau_{i,-j}\otimes \tau_{-j,j} - \tau_{-j,j}\otimes \tau_{i,-j} \right)
\end{align*}

Thus, in $\mathfrak{U}_{\mathbb{C}[[\hbar]]}\mathfrak{g}/\hbar\mathfrak{U}_{\mathbb{C}[[\hbar]]}\mathfrak{g}$, we have:
\begin{align*}
\overline{\hbar^{-1}\left(\Delta(\tau_{ij})-\Delta(\tau_{ij})^\circ\right)} = {} & \sum_{\substack{k=-n \\ \vert i \vert < \vert k \vert < \vert j \vert}}^n\left( (-1)^{(p(i)+p(k))(p(j)+p(k))}\overline{\tau}_{ik}\otimes \overline{\tau}_{kj} - \overline{\tau}_{kj}\otimes \overline{\tau}_{ik} \right)\\ 
&  + \frac{1}{2}\left( \overline{\tau}_{ii}\otimes \overline{\tau}_{ij} - \overline{\tau}_{ij}\otimes \overline{\tau}_{ii}  + \overline{\tau}_{ij}\otimes \overline{\tau}_{jj} - \overline{\tau}_{jj}\otimes \overline{\tau}_{ij}  \right)\\ 
&  - \delta_{i>0}\left(\overline{\tau}_{-i,j}\otimes \overline{\tau}_{i,-i} + (-1)^{p(j)}\overline{\tau}_{i,-i}\otimes \overline{\tau}_{-i,j} \right)\\ 
&  + \delta_{j<0}\left( (-1)^{p(i)}\overline{\tau}_{i,-j}\otimes \overline{\tau}_{-j,j} - \overline{\tau}_{-j,j}\otimes \overline{\tau}_{i,-j} \right)
\end{align*}
We next compute  $\delta(\overline{\tau}_{ij})$ using the isomorphism of Theorem \ref{defprop} and \eqref{cobr}. 
\begin{align*}
\delta(\overline{\tau}_{ij}) = {} & (-1)^{p(j)}\delta(\msE_{ji})\\
= {} & \sum_{\substack{k=-n \\ \vert i \vert < \vert k \vert < \vert j \vert}}^n(-1)^{p(j)+p(k)}\left ((-1)^{(p(i)+p(k))(p(j)+p(k))}\msE_{ki}\otimes \msE_{jk} - \msE_{jk}\otimes \msE_{ki} \right) \\
& {} - \frac{1}{2}(-1)^{p(j)}\left((-1)^{p(j)}\msE_{jj} - (-1)^{p(i)}\msE_{ii} \right)\otimes \msE_{ji} + \frac{1}{2}(-1)^{p(j)}\msE_{ji}\otimes \left((-1)^{p(j)}\msE_{jj} - (-1)^{p(i)}\msE_{ii} \right)\\
& {} - \frac{\delta_{j<0}}{2}(-1)^{p(j)}\msE_{j,-j}\otimes \msE_{-j,i} + \frac{\delta_{i>0}}{2}\msE_{-i,i}\otimes \msE_{j,-i}\\
& {} + \frac{\delta_{j<0}}{2}(-1)^{p(i)+p(j)}\msE_{-j,i}\otimes \msE_{j,-j} + \frac{\delta_{i>0}}{2}(-1)^{p(j)}\msE_{j,-i}\otimes \msE_{-i,i}\\
= {} & \overline{\hbar^{-1}\left(\Delta(\tau_{ij})-\Delta(\tau_{ij})^\circ \right)} 
\end{align*}
as needed. \end{proof} 

\section{Periplectic $q$-Brauer algebra}

In \cite{Mo}, D. Moon identified the centralizer of the action of $\mfp_n$ on the tensor space $\C_{n|n}^{\ot l}$. This centralizer is called the periplectic Brauer algebra in the literature: see \cite{Co,CP,CE1,CE2}.

Since $S$ is a solution of the quantum Yang-Baxter equation, we have a representation of $\mfU_q\mfp_n$ on $\C_q(n|n)$ via the assignment $t_{ij}\mapsto s_{ij}$ (where $S = \sum_{i,j=-n}^n s_{ij} \otimes E_{ij}$), and thus we also have a representation on each tensor power $\C_q(n|n)^{\ot l}$. In this section, we identify the centralizer of the action of $\mfU_q\mfp_n$ on $\C_q(n|n)^{\ot l}$ and call it the periplectic $q$-Brauer algebra. For the quantum group of type $Q$, this was done in \cite{Ol} and the centralizer of its action is called the Hecke-Clifford superalgebra.  Quantum analogs of the Brauer algebra were studied in \cite{M} where they appear as centralizers of the action of twisted quantized enveloping algebras $\mfU_q^{tw}\mfo_n$ and $\mfU_q^{tw}\mfsp_n$ on tensor representations (here, $\mfsp_n$ is the symplectic Lie algebra); see also \cite{We}. 

\begin{definition} The periplectic $q$-Brauer algebra $\mfB_{q,l}$ is the associative $\C(q)$-algebra generated by elements $\mst_i$ and $\msc_i$ for $1\le i\le l-1$ satisfying the following relations: \begin{gather}
(\mst_i-q)(\mst_i+q^{-1}) = 0, \;\;  \msc_i^2=0, \;\; \msc_i \mst_i = -q^{-1} \msc_i, \;\; \mst_i \msc_i = q \msc_i \;\; \text{ for } 1\le i \le l-1; \label{qB1} \\
\mst_i \mst_j = \mst_j \mst_i, \;\; \mst_i \msc_j = \msc_j \mst_i, \;\; \msc_i \msc_j = \msc_j \msc_i \;\; \text{ if } |i-j|\ge 2; \label{qB2} \\
\mst_i \mst_j \mst_i = \mst_j \mst_i \mst_j, \;\; \msc_{i+1} \msc_i \msc_{i+1} = -\msc_{i+1}, \;\; \msc_i \msc_{i+1} \msc_i = - \msc_i \;\; \text{ for } 1\le i \le l-2; \label{qB3} \\
\mst_i \msc_{i+1} \msc_i = -\mst_{i+1} \msc_i + (q-q^{-1}) \msc_{i+1} \msc_i, \;\;  \msc_{i+1} \msc_i \mst_{i+1}  = -\msc_{i+1} \mst_i + (q-q^{-1}) \msc_{i+1} \msc_i \label{qB4}
\end{gather}
\end{definition}
\begin{remark}
Setting $q=1$ in this definition yields the algebra $A_l$ from Definition 2.2 in \cite{Mo}.
\end{remark}

\begin{lemma}\label{lemma:varthetaepsilon}
Consider $\C(q)$ as purely odd $\mfU_q\mfp_n$-module. We have $\mfU_q\mfp_n$-module homomorphisms $\vartheta: \C_q(n|n)\ot\C_q(n|n) \rightarrow \C(q)$ and $\epsilon: \C(q) \rightarrow \C_q(n|n)\ot\C_q(n|n)$ given by $\vartheta(e_a \ot e_b) = \delta_{a,-b}(-1)^{p(a)}$ and $\epsilon(1) = \sum_{a=-n}^n e_a \ot e_{-a}$.
\end{lemma}

\begin{proof} It is enough to check that, for all the generators $t_{ij}$ of $\mfU_q\mfp_n$ and any tensor $\mbv \in \C_q(n|n)\ot\C_q(n|n)$, \begin{equation} \vartheta(t_{ij}(\mbv)) = t_{ij}(\vartheta(\mbv))\text{ and } \epsilon(t_{ij}(1)) = t_{ij}(\epsilon(1)). \label{eq:varthetaepsilon} \end{equation} Here is a brief sketch of some of the computations.

Using the formula for the coproduct, we have: \begin{align}
t_{ij}(e_a\ot e_{-a}) = {} &  \sum_{k=-n}^n (-1)^{(p(i)+p(k))(p(k)+p(j)) + (p(k)+p(j))p(a)} t_{ik}(e_a) \ot t_{kj}(e_{-a}) \label{eq:varthetaepsilonproduct}
\end{align}
This can be made more explicit using 
\begin{eqnarray*}
t_{ii}(e_a) &=& \sum_{b=-n}^n q^{\delta_{bi}(1-2p(i)) + \delta_{b,-i}(2p(i)-1)} E_{bb}(e_a); \\
t_{i,-i}(e_a) &=& (q-q^{-1}) \delta_{i>0} E_{-i,i}(e_a); \\
t_{ij}(e_a) & = &  (q-q^{-1}) (-1)^{p(i)} \msE_{ji}(e_a), \text{ if  }|i|\neq |j|.
\end{eqnarray*}
We obtain, for instance, \begin{align*}
t_{ii}(e_{a_1}\ot e_{a_2}) = {} & q^{\delta_{a_{1},i}(1-2p(i)) + \delta_{a_{1},-i}(2p(i)-1)} q^{\delta_{a_2,i}(1-2p(i)) + \delta_{a_2,-i}(2p(i)-1)} e_{a_1} \ot  e_{a_2}
\end{align*}
If $a_2 = -a_1=-a$, this simplifies to $e_{a} \ot e_{-a}$ and this allows us to check \eqref{eq:varthetaepsilon} quickly for $i=j$.

Furthermore, \begin{align*}
t_{i,-i}(e_{a}\ot e_{-a}) = {} & (-1)^{p(a)} \delta_{i>0} t_{ii}(e_a) \ot t_{i,-i}(e_{-a}) + \delta_{i>0} t_{i,-i}(e_a) \ot t_{-i,-i}(e_{-a}) 
\end{align*}
It follows that $t_{i,-i}\left(\sum_{a=-n}^n e_{a}\ot e_{-a}\right) = 0$, so the identity for  $\epsilon$  in\eqref{eq:varthetaepsilon} holds  for $j=-i$.

Suppose now that $a_1\neq -a_2$. Then
\begin{align*}
t_{i,-i}(e_{a_1}\ot e_{a_2}) = {} & \delta_{i>0} \delta(a_1=a_2=i) (q-q^{-1}) q e_{i} \ot e_{-i} \\
{} & + \delta_{i>0} \delta(a_1=a_2=i) (q-q^{-1}) q  e_{-i} \ot e_{i} \\
\end{align*}
Observe that $\vartheta(e_{i} \ot e_{-i} + e_{-i} \ot e_{i}) = 0$, so we have shown that $\vartheta(t_{i,-i}(e_{a_1}\ot e_{a_2})) = t_{i,-i}(\vartheta(e_{a_1}\ot e_{a_2}))$ and this proves \eqref{eq:varthetaepsilon} for $\vartheta$ when $j=-i$.

Next, we consider the case  $|i| \neq |j|$. To prove the identity  for $\epsilon$ in  \eqref{eq:varthetaepsilon}, we use again  \eqref{eq:varthetaepsilonproduct} and obtain that $\displaystyle t_{ij}\left(\sum_{a=-n}^n e_{a}\ot e_{-a}\right) = 0$ by considering subcases $i=\pm a$, $j=\pm a$, and $k=\pm a$. To show that \eqref{eq:varthetaepsilon} holds for $\vartheta$ we also proceed with case-by-case verification. The case   $a_1,a_2\not\in \{\pm i, \pm j\}$ is immediate.  If $a_1\in \{\pm i, \pm j\}$,  $a_2 \not\in \{\pm i, \pm j\}$, and$a_1\neq -a_2$, then
\begin{align*}
t_{ij}(e_{a_1}\ot e_{a_2}) = {} &  (q-q^{-1})^2 (-1)^{(p(i)+p(a_2))(p(a_2)+p(j)) + (p(a_2)+p(j))p(a_1)} (-1)^{p(i)+p(a_2)} \msE_{a_2i}(e_{a_1}) \ot \msE_{ja_2}(e_{a_2}) \\
& + (q-q^{-1}) (-1)^{p(i)} \msE_{ji}(e_{a_1}) \ot  E_{a_2a_2}(e_{a_2}).
\end{align*}
 This shows that $\vartheta(t_{ij}(e_{a_1} \otimes e_{a_2}))=0 = t_{ij}(\vartheta(e_{a_1} \otimes e_{a_2}))$. Similarly, we obtain the desired identity in the other cases. \end{proof}

By composing $\vartheta$ and $\epsilon$, we obtain a $\mfU_q\mfp_n$-module homomorphism $\epsilon\circ\vartheta:  \C_q(n|n)^{\ot 2} \rightarrow \C_q(n|n)^{\ot 2}$. In terms of elementary matrices, this linear map is given by $\sum_{a,b=-n}^n (-1)^{p(a)p(b)} E_{ab} \otimes E_{-a,-b}$, which we abbreviate by $\mfc$. The super-permutation operator $P$ on $ \C_q(n|n)^{\ot 2}$ is given by $P = \sum_{a,b=-n}^n (-1)^{p(b)}E_{ab} \ot E_{ba}$, so $\mfc = P^{(\pi\circ\mathrm{st})_2}$ where $(\pi\circ\mathrm{st})_2$ stands for the map $\pi\circ\mathrm{st}$ applied to the second tensor in the previous formula for $P$.

We can extend $\mfc$ to a $\mfU_q\mfp_n$-module homomorphism $\mfc_i:  \C_q(n|n)^{\ot l} \rightarrow \C_q(n|n)^{\ot l}$ for $1\le i\le l-1$ by applying $\mfc$ to the $i^{th}$ and $(i+1)^{th}$ tensors. 

The linear map $\C_q(n|n)^{\ot l} \rightarrow \C_q(n|n)^{\ot l}$ given by $P_i S_{i,i+1}$ where $P_i$ is the super-permutation operator acting on the $i^{th}$ and $(i+1)^{th}$ tensors is also a $\mfU_q\mfp_n$-module homomorphism: this is a consequence of the fact that $S$ is a solution of the quantum Yang-Baxter relation. 

\begin{proposition}
The tensor superspace $\C_q(n|n)^{\ot l}$ is a module over $\mfB_{q,l}$ if we let $\mst_i$ act as $P_iS_{i,i+1}$ and $\msc_i$ act as $\mfc_i$.
\end{proposition}

\begin{proof} That the linear operators $P_i S_{i,i+1}$ satisfy the braid relation (the first relation in \eqref{qB3}) is a consequence of the fact that $S$ is a solution of the quantum Yang-Baxter relation. The relations \eqref{qB2} for the operators $P_i S_{i,i+1}$ and $\mfc_i$ can be easily verified.  As for the other relations, they can be checked via direct computations. It is enough to check the relations \eqref{qB1} on $ \C_q(n|n)^{\ot 2}$ and the relations \eqref{qB4} on $ \C_q(n|n)^{\ot 3}$. We briefly sketch some of those computations below. 

First, note that $\mfc P =  - \mfc$ and $P \mfc = \mfc$. Also, we easily obtain the following:
\begin{align*}
\mfc \left((q-1)\sum_{i=1}^n E_{ii} \otimes E_{ii}\right) = {} & \mfc \left((q^{-1}-1)\sum_{i=1}^n E_{-i,-i} \otimes E_{-i,-i}\right) = 0,\\
\mfc \left((q-1)\sum_{i=1}^n E_{ii} \otimes E_{-i,-i}\right) = {} & (q-1) \sum_{a=-n}^n \sum_{b=1}^n E_{ab} \otimes E_{-a,-b},\\
\mfc  \left((q^{-1}-1)\sum_{i=1}^n E_{-i,-i} \otimes E_{ii}\right) = {} & (q^{-1}-1) \sum_{a=-n}^n \sum_{b=-n}^{-1} (-1)^{p(a)} E_{ab} \otimes E_{-a,-b},\\
\mfc  \left( \sum_{i=-n}^{-1} E_{i,-i} \otimes E_{-i,i} \right) = {} & - \sum_{a=-n}^n \sum_{b=1}^{n}  E_{ab} \otimes E_{-a,-b},\\
\mfc  \left( \sum_{1\leq |j|<|i|\leq n} (-1)^{p(j)}\msE_{ij} \otimes E_{ji} \right) = {} & \sum_{a=-n}^n \sum_{1\leq |j|<|i|\leq n} (-1)^{p(a)(p(i)+1)+p(j)}  E_{a,-i} \otimes E_{-a,i} = 0.
\end{align*}

Therefore, we have that $\mfc(S-1)=(q^{-1}-1)\mfc$, hence $\mfc S = q^{-1}\mfc$. Now using that $\mfc = -\mfc P$, we obtain the third relation in \eqref{qB1}. Similarly, we prove  $(S-1)\mfc = (q-1)\mfc$, and then using $P\mfc = \mfc$, we obtain  the fourth relation in \eqref{qB1}.

For the remaining relations we use the following formula:
\begin{align*}
PS = {} &  \sum_{i,j=-n}^n (-1)^{p(j)} E_{ij} \ot E_{ji} 
 + (q-1) \sum_{i=1}^n \left(  E_{-i,i} \ot E_{i,-i} \right)  \\
& + (q-1) \sum_{i=1}^n \left(  E_{ii} \ot E_{ii} \right)  
 - (q^{-1}-1) \sum_{i=1}^n  \left(  E_{i,-i} \ot E_{-i,i} \right)  \\
& - (q^{-1}-1) \sum_{i=1}^n  \left(E_{-i,-i} \ot E_{-i,-i} \right)  
 + (q-q^{-1}) \sum_{i=-n}^{-1} \left(  E_{-i,-i} \ot E_{ii} \right)   \\
& + (q-q^{-1}) \sum_{|j|<|i|} \left(E_{jj} \ot E_{ii} \right) 
 + (q-q^{-1}) \sum_{|j|<|i|} \left( (-1)^{p(i)p(j)} E_{ji} \ot E_{-j,-i} \right) 
\end{align*}
 \end{proof}

As mentioned after the definition of $\mfB_{q,l}$, the module structure given in the previous proposition commutes with the action of $\mfU_q(\mfp_n)$ on $\C_q(n|n)^{\ot l}$. We thus have algebra homomorphisms $$\mfB_{q,l} \longrightarrow \End_{\mfU_q(\mfp_n)}(\C_q(n|n)^{\ot l}) \text{ and }  \mfU_q(\mfp_n) \longrightarrow \End_{\mfB_{q,l}}(\C_q(n|n)^{\ot l}).$$ The main theorem of this section states that $\mfB_{q,l}$ is the full centralizer of the action of $\mfU_q(\mfp_n)$ on $\C_q(n|n)^{\ot l}$ when $n\ge l$.  

\begin{theorem}\label{thm:Bqcent}
The map $\mfB_{q,l} \longrightarrow \End_{\mfU_q(\mfp_n)}(\C_q(n|n)^{\ot l})$ is surjective and it is injective when $n\ge l$. 
\end{theorem}

\begin{proof} This is a $q$-analogue of Theorem 4.5 in \cite{Mo}. The proof follows the lines of the proof of Theorem 3.28 in \cite{BGJKW}.

Recall that  $\mcA = \C[q,q^{-1}]_{(q-1)}$ is the localization of $\C[q,q^{-1}]$ at the ideal generated by $q-1$. The algebra $\mfU_{\mcA}\mfp_n$ was defined at the beginning of Section \ref{limquant} and it acts on $\C_{\mcA}(n|n)^{\otimes l}$. Let's abbreviate it $\wt{\mfU}$ for the moment. Let $\End_{\wt{\mfU}}(\C_{\mcA}(n|n)^{\otimes l})$ be the $\mcA$-subalgebra of $\End_{\mcA}(\C_{\mcA}(n|n)^{\otimes l})$ that consists of all the $\mcA$-endomorphisms of $\C_{\mcA}(n|n)^{\otimes l}$ that commute with the action of $\wt{\mfU}$. 

Let $\mfB_{q,l}(\mcA)$ be the $\mcA$-associative subalgebra of $\mfB_{q,l}$ generated by $\mst_i$ and $\msc_i$ for all $i=1,\ldots,l-1$. Theorem 5.5 will follow from the statement that the $\mcA$-homomorphism $$\psi: \mfB_{q,l}(\mcA) \longrightarrow \End_{\wt{\mfU}}(\C_{\mcA}(n|n)^{\otimes l}) $$ is surjective and is an isomorphism whenever $n\ge l$.

Let $A_{l}$ be the algebra given in Definition 2.2 in [Mo]. Proposition \ref{prop:q1Brauer} below gives use an isomorphism $\rho: A_{l} \longrightarrow (\mcA/(q-1)\mcA) \otimes_{\mcA} \mfB_{q,l}(\mcA)$ which fits within the following diagram (see the proof of Theorem 3.28 in \cite{BGJKW}). 

\[\xymatrix@C=20mm{
    A_{l} \ar[r]^-{\rho}\ar@{->>}[dr]& (\mcA/(q-1)\mcA) \otimes_{\mcA} \mfB_{q,l}(\mcA) \ar[r]^-{\id \ot \psi}&
        (\mcA/(q-1)\mcA)  \otimes_{\mcA} \End_{\wt{U}}(\C_{\mcA}(n|n)^{\otimes l})\ar[dl]\ar@{^{(}->}[d]\\
    &\End_{\mfp_n}(\C(n|n)^{\otimes l})\ar@{^{(}->}[r]&\End_{\mathbb{C}}(\C(n|n)^{\otimes l})
}\]

The rest of the proof can proceed as in  \cite{BGJKW}), using Theorem 4.5 in \cite{Mo} along with Lemma 3.27 in \cite{BGJKW}, which can be applied in the present situation. \end{proof}

\begin{proposition}\label{prop:q1Brauer}
The quotient algebra $\mfB_{q,l}(\mcA)/(q-1)\mfB_{q,l}(\mcA)$ is isomorphic to the algebra $A_l$ given in Definition 2.2 in \cite{Mo}.
\end{proposition}

\begin{proof} It follows immediately from the definitions of both $A_l$ and $\mfB_{q,l}(\mcA)$ that we have a surjective algebra homomorphism $A_l \onto \mfB_{q,l}(\mcA)/(q-1)\mfB_{q,l}(\mcA)$. That it is injective can be proved as in the proof of Proposition 3.21 in \cite{BGJKW} using Theorem 4.1 in \cite{Mo}. \end{proof}

The $q$-Schur superalgebras of type $Q$ were introduced in \cite{BGJKW} and \cite{DuWa1,DuWa2}. Considering \textit{loc. cit.} and the earlier work on $q$-Schur algebras for $\mfgl_n$ (see for instance \cite{Do}), the following definition is natural. 

\begin{definition}
The $q$-Schur superalgebra $S_q(\mfp_n,l)$ of type $P$ is the centralizer of the action of $\mfB_{q,l}$ on $\C_q(n|n)^{\ot l}$, that is, $S_q(\mfp_n,l) = \End_{\mfB_{q,l}}(\C_q(n|n)^{\ot l})$.
\end{definition}

We have an algebra homomorphism $\mfU_q(\mfp_n) \longrightarrow S_q(\mfp_n,l)$: it is an open question whether or not this map is surjective. We also have an algebra homomorphism $\mfB_{q,l} \lra \End_{S_q(\mfp_n,l)}(\C_q(n|n)^{\ot l})$ and it is natural to expect that it should be an isomorphism, perhaps under certain conditions on $n$ and $l$.

\noindent
Department of Mathematics, University of Texas at Arlington 
\\ Arlington, TX 76021, USA\\
saber.ahmed\@@mavs.uta.edu

\noindent
Department of Mathematics, University of Texas at Arlington 
\\ Arlington, TX 76021, USA\\
grandim\@@uta.edu

\noindent
University of Alberta, Department of Mathematical and Statistical Sciences, CAB 632\\
Edmonton, AB T6G 2G1, Canada\\
nguay\@@ualberta.ca


\begin{thebibliography}{100}



\bibitem[AGG]{AGG} S. Ahmed, D. Grantcharov, N. Guay, \emph{Quantized enveloping superalgebra of type $P$}, .

\bibitem[BDEA$^+1$]{BDEA^+1} M. Balagovic, Z. Daugherty, I. Entova-Aizenbud, I. Halacheva, J. Hennig, M. S. Im, G. Letzter, E. Norton, V. Serganova, C. Stroppel, \emph{The affine $VW$ supercategory},  Selecta Math. (N.S.) \textbf{26} (2020), no. 2, Paper No. 20, 42 pp., arXiv:1801.04178.

\bibitem[BDEA$^+2$]{BDEA^+2} M. Balagovic, Z. Daugherty, I. Entova-Aizenbud, I. Halacheva, J. Hennig, M. S. Im, G. Letzter, E. Norton, V. Serganova, C. Stroppel,  \emph{Translation functors and decomposition numbers for the periplectic Lie superalgebra $\mathfrak{p}(n)$},  Math. Res. Lett. \textbf{26} (2019), no. 3, 643--710, arXiv:1610.08470.

\bibitem[BGJKW]{BGJKW} G. Benkart, N. Guay, J. H. Jung, S.-J. Kang, S. Wilcox, \emph{Quantum walled Brauer-Clifford superalgebras}, J. Algebra \textbf{454} (2016), 433--474.

\bibitem[CE1]{CE1}  K. Coulembier, M. Ehrig, \emph{The periplectic Brauer algebra II: Decomposition multiplicities},   J. Comb. Algebra \textbf{2} (2018), no. 1, 19--46.

\bibitem[CE2]{CE2} K. Coulembier, M. Ehrig, \emph{The periplectic Brauer algebra III: The Deligne category},  Algebr Represent Theory (2020). https://doi.org/10.1007/s10468-020-09976-8.
 
\bibitem[ChGu]{ChGu} H. Chen, N. Guay, \emph{Twisted affine Lie superalgebra of type $Q$ and quantization of its enveloping superalgebra}, Math. Z., \textbf{272} (2012), no.1, 317--347.

\bibitem[Co]{Co}  K. Coulembier, \emph{The periplectic Brauer algebra},  Proc. Lond. Math. Soc. (3) \textbf{117} (2018), no. 3, 441--482.

\bibitem[CP]{CP} C. Chen, Y. Peng, \emph{Affine periplectic Brauer algebras}, J. Algebra \textbf{501} (2018), 345--372.

\bibitem[DHIN]{DHIN}  Z. Daugherty, I. Halacheva, M.S. Im, and E. Norton, \emph{On calibrated representations of the degenerate affine periplectic Brauer algebra}, arXiv:1905.05148.

\bibitem[Do]{Do} S. Donkin, \emph{The $q$-Schur Algebra},
London Mathematical Society Lecture Note Series, \textbf{253}, Cambridge University Press, Cambridge, 1998.

\bibitem[DuWa1]{DuWa1} J. Du, J. Wan, \emph{Presenting queer Schur superalgebras}, Int. Math. Res. Not. IMRN 2015, no. 8, 2210--2272.

\bibitem[DuWa2]{DuWa2} J. Du, J. Wan, \emph{The queer $q$-Schur superalgebra}, J. Aust. Math. Soc. (2018), \textbf{105}, no.3, 316--346.

\bibitem[EAS1]{EAS1}   Inna Entova-Aizenbud, Vera Serganova, \emph{Deligne categories and the periplectic Lie superalgebra},  arXiv:1807.09478.

\bibitem[EAS2]{EAS2}  Inna Entova-Aizenbud, Vera Serganova, \emph{Kac-Wakimoto conjecture for the periplectic Lie superalgebra}, arXiv:1905.04712.

\bibitem[FRT]{FRT} L. Faddeev, N. Reshetikhin, L. Takhtajan, \emph{Quantization of Lie groups and Lie algebras} (Russian),  Algebra i Analiz  \textbf{1}  (1989),  no. 1, 178--206;  translation in  Leningrad Math. J.  \textbf{1}  (1990),  no. 1, 193--225.

\bibitem[FSS]{FSS} L. Frappa, P. Sorba, A. Sciarrino, \emph{Deformation of the strange superalgebra $\tilde{P}(n)$}, J. Phys. A: Math. Gen. \textbf{26} (1993) 661--665.

\bibitem[GJKK]{GJKK} D. Grantcharov, J.H. Jung., S.-J. Kang, M. Kim, \emph{Highest weight modules over quantun queer superalgebra $U_q(\mfq(n))$},  Comm. Math. Phys. \textbf{296}  (2010),  no. 3, 827--860.

\bibitem[GJKKK]{GJKKK} D. Grantcharov, J.H. Jung., S.-J. Kang, M. Kashiwara, M. Kim, \emph{Quantum Queer Superalgebra and Crystal Bases},  Proc. Japan Acad. Ser. A Math. Sci. \textbf{86} (2010), no. 10, 177--182.

\bibitem[HIR]{HIR} C. Hoyt, M.S. Im, S. Reif, Denominator identities for the periplectic Lie superalgebra, J. Algebra, \textbf{567} (2021), 459--474.


\bibitem[IN]{IN} M.S. Im, E. Norton, Irreducible calibrated representations of periplectic Brauer algebras and hook representations of the symmetric group, arXiv:1906.07472.

\bibitem[IRS]{IRS} M. S. Im, S. Reif, V. Serganova, Grothendieck rings of periplectic Lie superalgebras, to appear in Math. Res. Lett.


\bibitem[K]{Kac}  V. Kac, \emph{Lie superalgebras}, Adv. Math. \textbf{26} (1977), no. 1, 8--96.
 
\bibitem[KT]{KT}  J. Kujawa, B. Tharp, \emph{The marked Brauer category}, J. Lond. Math. Soc. \textbf{95} (2017), no. 2, 393--413.

\bibitem[LeSh]{LeSh} D. Leites, A. Shapovalov, \emph{Manin-Olshansky triples for Lie superalgebras}, J. Nonlinear Math. Phys. \textbf{7} (2000), no. 2, 120--125. 

\bibitem[M]{M} A. Molev, \emph{A new quantum analog of the Brauer algebra}, Quantum groups and integrable systems, Czechoslovak J. Phys. \textbf{53} (2003), no. 11, 1073--1078.

\bibitem[Mo]{Mo} D. Moon, \emph{Tensor product representations of the Lie superalgebra $\mfp(n)$and their centralizers},  Comm. Algebra \textbf{31} (2003), no. 5, 2095--2140. 

\bibitem[Na1]{Na1} M. Nazarov, \emph{Yangians of the "strange'' Lie superalgebras}, Quantum groups (Leningrad, 1990), Lecture Notes in Math. \textbf{1510},  pp. 90-–97, Springer, Berlin, 1992. 

\bibitem[Na2]{Na2} M. Nazarov, \emph{Yangian of the queer Lie superalgebra}, Comm. Math. Phys. \textbf{208} (1999), no. 1, 195--223.

\bibitem[Ol]{Ol} G. Olshanski, \emph{Quantized universal enveloping superalgebra of type $Q$ and a super-extension of the Hecke algebra},  Lett. Math. Phys.  \textbf{24}  (1992),  no. 2, 93--102.

\bibitem[Ser]{Ser}  V. Serganova, \emph{On representations of the Lie superalgebra $\mathfrak{p}(n)$}, J. Algebra \textbf{258} (2002), 615--630.

\bibitem[We]{We} H. Wenzl, \emph{A $q$-Brauer algebra}, J. Algebra \textbf{358} (2012), 102--127.
\end{thebibliography}
\end{document}